\documentclass[12pt,reqno]{amsart}
\usepackage{amsmath, amssymb, amsthm}
\usepackage{url}
\usepackage{mathrsfs}
\usepackage[ansinew]{inputenc}
\usepackage[breaklinks]{hyperref}
\usepackage{fancyhdr}

\usepackage{multirow}

\allowdisplaybreaks

\newcommand{\pFq}[5]{\ensuremath{{}_{#1}{}\phi_{#2} \left( \genfrac{}{}{0pt}{}{#3}
{#4} ; {#5} \right)}}

\renewcommand{\(}{\left\(}
\renewcommand{\)}{\right\)}
\renewcommand{\[}{\left\[}
\renewcommand{\]}{\right\]}

\numberwithin{equation}{section}
 \theoremstyle{plain}
\newtheorem{theorem}{Theorem}[section]
\newtheorem{lemma}[theorem]{Lemma}
\newcommand{\sm}{\left(\begin{smallmatrix}}
\newcommand{\esm}{\end{smallmatrix}\right)}

\newtheorem{corollary}[theorem]{Corollary}

   \makeatletter
\def\proof{\@ifnextchar[{\@oproof}{\@nproof}}
\def\@oproof[#1][#2]{\trivlist\item[\hskip\labelsep\textit{#2 \textbf{Proof of}\
#1.}~]\ignorespaces}
\def\@nproof{\trivlist\item[\hskip\labelsep\textit{Proof.}~]\ignorespaces}

\makeatother

\usepackage{color}
\usepackage{amsmath}

\definecolor{blue}{rgb}{0,0,1}
\definecolor{red}{rgb}{1,0,0}
\definecolor{green}{rgb}{0,.6,.2}
\definecolor{purple}{rgb}{1,0,1}

\long\def\red#1\endred{{\color{red}#1}}
\long\def\blue#1\endblue{{\color{blue}#1}}
\long\def\purple#1\endpurple{{\color{purple}#1}}
\long\def\green#1\endgreen{{\color{green}#1}}

\begin{document}
\title[Rank, two-color partitions and Mock theta function]{Rank, two-color partitions and Mock theta function}
\author{George E. Andrews}
\address{Department of Mathematics, The Pennsylvania State University, University Park, PA 16802, U.S.A.}
\email{gea1@psu.edu}

\author{Rahul Kumar}
\address{Department of Mathematics, Indian Institute of Technology, Roorkee-247667, Uttarakhand, India}
\email{rahul.kumar@ma.iitr.ac.in} 


\maketitle
\pagenumbering{arabic}
\pagestyle{headings}
\begin{abstract}
In this paper, we establish that the number of partitions of a natural number with positive odd rank is equal to the number of two-color partitions (red and blue), where the smallest part is even (say $2n$) and all red parts are even and lie within the interval $(2n,4n]$. This led us to derive a new representation for the third order mock theta function $f_3(q)$ and an analogue of the fundamental identity for the smallest part partition function Spt$(n)$, both of which are of significant interest in their own right. We also consider the odd smallest part version of the above two-color partition, whose generating function involves another third order mock theta function $\phi_3(q)$.

\end{abstract}

\section{Introduction}

In their paper \cite{ady}, Dixit, Yee and the first author found surprising partition identities arising from new results for two of the third order mock theta functions $\omega_3(q)$ and $\nu_3(q)$, originally studied by Ramanujan and Watson. One of the key identities for $\omega_3(q)$, established by Dixit, Yee, and the first author, is \cite[Theorem 3.1]{ady}
\begin{align}\label{omega3}
\omega_3(q)=\sum_{n=1}^\infty\frac{q^{n-1}}{(1-q^n)(q^{n+1};q)_n(q^{2n+2};q^2)_\infty},
\end{align}
where the mock theta function $\omega_3(q)$ is defined as \cite[p.~15]{rlnb} or \cite[p.~62]{watson}
\begin{align*}
\omega_3(q):=\sum_{n=0}^\infty\frac{q^{2n(n+1)}}{(q;q^2)_{n+1}},
\end{align*}
and, 
\begin{align*}
(A;q)_N&=(1-A)(1-Aq)\cdots(1-Aq^{N-1}),\qquad 1\leq N<\infty.
\end{align*}
We will also be using the following notation
\begin{align*}
(A;q)_\infty&=\lim_{N\to\infty}(A;q)_N.
\end{align*}

A very similar result holds for the mock theta function $\nu_3(-q)$, as shown in \cite[Theorem 4.1]{ady}. These kinds of identities are always desirable as they offer clear partition-theoretic interpretations. For example, it follows easily from \eqref{omega3} that $q\omega(q)$ is the generating function for partitions in which each odd part is less than twice the smallest part. Moreover, very recently, it has been established by Andrews and  Bachraoui \cite{ab} that $q\omega_3(q)$ and $\nu_3(-q)$ are generating functions for two-color partitions in which the even or the odd parts may occur only in one color. 

The aforementioned work  of Andrews, Dixit and Yee \cite{ady} was inspired by the fundamental identity for the smallest part partition function, spt$(n)$, introduced in \cite{andrewsspt}:
\begin{align}\label{andrews spt function}
\sum_{n=1}^\infty\mathrm{spt}(n)q^n&=\sum_{n=1}^\infty\frac{q^{n}}{(1-q^{n})^2(q^{n+1};q)_\infty}\nonumber\\
&=\frac{1}{(q;q)_\infty}\sum_{n=1}^\infty\frac{nq^{n}}{1-q^{n}}+\frac{1}{(q;q)_\infty}\sum_{n=1}^\infty\frac{(-1)^n(1+q^n)}{(1-q^{n})^2}q^{\frac{3n}{2}(n+1)},
\end{align}
where spt$(n)$ is the number of smallest parts in all the partitions of $n$.

Missing from these considerations is the principal third order mock theta function \cite[pp. 127--131]{rlnb}
\begin{align}
f_3(q):=\sum_{n=0}^\infty\frac{q^{n^2}}{(-q;q)_n^2}.\nonumber
\end{align}
The mock theta function $f_3(q)$ has been the subject of numerous important research papers due to its remarkable properties (see, for example, \cite{ad, andrewsAJM, bo, za, zw} to name a few).

In this paper, we shall prove the following result for $f_3(q)$.
\begin{theorem}\label{new representation}
The following identity holds:
\begin{align}
f_3(q)=\frac{1}{(q;q)_\infty}-4\sum_{n=1}^\infty\frac{q^{2n}}{(q^{2n};q^2)_{n+1}(q^{2n+1};q)_\infty}.\nonumber
\end{align}
\end{theorem}

To see the partition theoretic interpretation of this, we let $N_o^+(n)$ denote the number of partitions of $n$ with positive odd rank (the rank of a partition is the largest part minus the number of parts), and we let $G(n)$ denote the number of partitions into red and blue parts, where the smallest part is even (say $2m$) and all red parts are even and lie within the interval $(2m,4m]$. For example, we have $G(8)=7$ as we have:
\begin{align*}
8_b,\ 6_b+2_b,\ 4_b+4_b,\ 4_b+2_b+2_b,\ 2_b+2_b+2_b+2_b,\ 3_b+3_b+2_b,\ 4_r+2_b+2_b.
\end{align*}

Similarly, observe that there are exactly $7$ partitions with a positive odd rank:
\begin{align*}
8,\ 7+1,\ 6+1+1,\ 5+3,\ 5+1+1+1,\ 4+3+1,\ 4+2+2.
\end{align*}

Hence, we have $N_o^+(8)=7=G(8)$.

Our next result establishes that this is not merely a coincidence, but that this equality holds for every natural number $n$.

\begin{theorem}\label{G(n)=N(n)}
For $n\geq1$,
\begin{align*}
N_o^+(n)=G(n).
\end{align*}
\end{theorem}

Our proof of these results will rely on the following analog of \eqref{andrews spt function}.
\begin{theorem}\label{spt theorem}
We have the following $q$-series identity:
\begin{align}\label{spt theorem eqn}
&\sum_{n=1}^\infty\frac{q^{2n}}{(1-q^{2n})^2(-q^{n+1};q)_n(q^{n+1};q)_\infty}\nonumber\\
&=\frac{1}{(q;q)_\infty}\sum_{n=1}^\infty\frac{nq^{2n}}{1-q^{2n}}+\frac{1}{(q;q)_\infty}\sum_{n=1}^\infty\frac{(-1)^n(1+q^n)}{(1-q^{2n})^2}q^{\frac{3n}{2}(n+1)}.
\end{align}
\end{theorem}

We conclude our introduction by presenting some complementary results, which follow from the results discussed above.  To state them, let $J(n)$ denote Fine's numbers, which are defined as \cite[p.~56]{fine}
\begin{align}\label{fine numbers}
\sum_{n=1}^\infty J(n)q^n=\sum_{n=1}^\infty \frac{(-1)^nq^{n(3n+1)/2}}{1+q^n}.
\end{align} 

\begin{corollary}\label{Ne(n)}
Let $N_e(n)$ denotes the number of partitions of $n$ with even rank. Then,
\begin{align}\label{even rank generating function}
\sum_{n=1}^\infty N_e(n)q^n=\frac{1}{(q;q)_\infty}-2\sum_{n=1}^\infty\frac{q^{2n}}{(1-q^{2n})(-q^{n+1};q)_n(q^{n+1};q)_\infty}.
\end{align}
Also, if $J(n)$ is as defined in \eqref{fine numbers}, then
\begin{align}\label{fine numbers result}
\sum_{n=1}^\infty J(n)q^n=-\sum_{n=1}^\infty\frac{(q;q)_nq^{2n}}{(1-q^{2n})(-q^{n+1};q)_n}.
\end{align} 
\end{corollary}

We consider the odd smallest part version of $G(n)$ in Section \ref{odd smallest part} below. Its generating function involves another third order mock theta function $\phi_3(q)$.

This paper is organized as follows. In Section \ref{pre}, we collect relevant results from the literature which will be needed for our proofs. Section \ref{new representation for f} is dedicated to the proof of Theorem \ref{new representation}. In Section \ref{partition ranks and two-color}, we establish Theorem \ref{G(n)=N(n)} and derive Corollary \ref{Ne(n)}. Section \ref{analog of spt function} focuses on proving Theorem \ref{spt theorem}. The odd smallest part variation of $G(n)$ is considered in Section \ref{odd smallest part}. Finally, in Section \ref{concluding remarks}, we conclude the paper with some closing remarks.

\section{Preliminaries}\label{pre}

The following definition of the basic hypergeometric series ${}_{r+1}\phi_r$ will be used through the paper:
\begin{align*}
\pFq{r+1}{r}{a_1,a_2,\cdots,a_{r+1}}{b_1,b_2,\cdots,b_{r}}{q,z}:=\sum_{n=0}^\infty\frac{(a_1;q)_n(a_2;q)_n\cdots(a_{r+1};q)_n}{(q;q)_n(b_1;q)_n(b_2;q)_n\cdots(b_{r};q)_n}z^n.
\end{align*}

We will require the following ${}_{10}\phi_9$-transformation \cite[Equation (2.9)]{andrewsstacks}

\begin{align}\label{andrewsstacks}
&\lim_{N\to\infty}\pFq{10}{9}{a,q\sqrt{a},-q\sqrt{a},b,r_1,-r_1,r_2,-r_2,q^{-N},-q^{-N};q}{\sqrt{a},-\sqrt{a},\frac{aq}{b},\frac{aq}{r_1},-\frac{aq}{r_1},\frac{aq}{r_2},-\frac{aq}{r_2},aq^{N+1},-aq^{N+1}}{-\frac{a^3q^{3+2N}}{br_1^2r_2^2}}\nonumber\\
&=\frac{(a^2q^2;q^2)_\infty(a^2q^2/(r_1^2r_2^2);q^2)_\infty}{(a^2q^2/r_1^2;q^2)_\infty(a^2q^2/r_2^2;q^2)_\infty}\sum_{n=0}^\infty\frac{(r_1^2;q^2)_n(r_2^2;q^2)_n(-aq/b;q)_{2n}}{(q^2;q^2)_n(a^2q^2/b^2;q^2)_n(-aq,q)_{2n}}\left(\frac{a^2q^2}{r_1^2r_2^2}\right)^n.
\end{align}

Additionally, we need the following $q$-series identity from \cite[p.~141, Theorem 1]{andrews rlnb}
\begin{align}\label{4parameter}
\sum_{n=0}^\infty\frac{(B;q)_n(-Abq;q)_nq^n}{(-aq;q)_n(-bq;q)_n}&=-\frac{(B;q)_\infty(-Abq;q)_\infty}{a(-bq;q)_\infty(-aq;q)_\infty}\sum_{m=0}^\infty\frac{(A^{-1};q)_m(Abq/a)^m}{(-B/a;q)_{m+1}}\nonumber\\
&\qquad+(1+b)\sum_{m=0}^\infty\frac{(-a^{-1};q)_{m+1}(-ABq/a;q)_m(-b)^m}{(-B/a;q)_{m+1}(Abq/a;q)_{m+1}}.
\end{align}

Finally, we require the following special case of the above transformation, but with the analytic continuation in the parameter $b$.
\begin{lemma}\label{analytic continuation}
For $|b|<\frac{1}{|q^2|}$, we have
\begin{align*}
\sum_{n=0}^\infty\frac{(q^2;q^2)_nq^{2n}}{(-q;q^2)_n(-bq^{2};q^2)_n}&=\frac{(1+q)}{1+q^3}-\frac{q(q^2;q^2)_\infty}{(-q;q^2)_\infty(-bq^2;q^2)_\infty}\sum_{m=0}^\infty\frac{(-1)^mb^mq^{m^2+2m}}{(-q^3;q^2)_{m+1}}\\
&\quad+(1+q)(1-q^2)\sum_{m=1}^\infty\frac{(-b)^mq^{2m+1}}{(1+q^{2m+1})(1+q^{2m+3})}.
\end{align*}
\end{lemma}

\begin{proof}
We first replace $q$ by $q^2$ and then set $B=q^2$, $a=1/q$ and let $A\to0$ in \eqref{4parameter} and use the fact
\begin{align*}
\lim_{A\to0}(A^{-1};q^2)_mA^m=(-1)^mq^{m(m-1)},
\end{align*}
to deduce that
\begin{align}\label{before ac}
\sum_{n=0}^\infty\frac{(q^2;q^2)_nq^{2n}}{(-q;q^2)_n(-bq^{2};q^2)_n}&=-\frac{q(q^2;q^2)_\infty}{(-q;q^2)_\infty(-bq^2;q^2)_\infty}\sum_{m=0}^\infty\frac{(-1)^mb^mq^{m^2+2m}}{(-q^3;q^2)_{m+1}}\nonumber\\
&\quad+(1+b)\sum_{m=0}^\infty \frac{(-q;q^2)_{m+1}}{(-q^3;q^2)_{m+1}}(-b)^m\nonumber\\
&=-\frac{q(q^2;q^2)_\infty}{(-q;q^2)_\infty(-bq^2;q^2)_\infty}\sum_{m=0}^\infty\frac{(-1)^mb^mq^{m^2+2m}}{(-q^3;q^2)_{m+1}}\nonumber\\
&\quad+(1+b)(1+q)\sum_{m=0}^\infty\frac{(-b)^m}{1+q^{2m+3}}.
\end{align}
We wish to have a range of $b$ that allows us to take $b=1$  in the above equation. However, this is not directly permissible because the second series on the right-hand side becomes divergent in this case. To address this, we require analytic continuation of the above result. To that end, we rewrite the second series as
\begin{align*}
(1+b)\sum_{m=0}^\infty\frac{(-b)^m}{1+q^{2m+3}}&=\sum_{m=0}^\infty\frac{(-1)^mb^m}{1+q^{2m+3}}+\sum_{m=0}^\infty\frac{(-1)^mb^{m+1}}{1+q^{2m+3}}\\
&=\frac{1}{1+q^{3}}+\sum_{m=1}^\infty\frac{(-1)^mb^m}{1+q^{2m+3}}-\sum_{m=1}^\infty\frac{(-1)^mb^m}{1+q^{2m+1}}\\
&=\frac{1}{1+q^{3}}+\sum_{m=1}^\infty\frac{(-b)^mq^{2m+1}(1-q^2)}{(1+q^{2m+1})(1+q^{2m+3})}.
\end{align*}
Now observe that the right-hand side is an analytic function of $b$ in the region $|b|<\frac{1}{|q^2|}$. Therefore, it provides the analytic continuation of the second series on the right-hand side of \eqref{before ac}.  Since the other terms in \eqref{before ac} are also analytic in this region, this completes the proof of the result.
\end{proof}

\medskip
\section{Analogue of the spt-function identity}\label{analog of spt function}
We begin with providing the following lemma. 
\begin{lemma}\label{closed form lemma}
\begin{align*}
\sum_{n=0}^\infty\frac{(-1)^nq^{2n+1}}{1+q^{2n+1}}=\frac{1}{4}-\frac{1}{4}\frac{(q;q)_\infty^2}{(-q;q)_\infty^2}.
\end{align*}
\end{lemma}
\begin{proof}
Replacing $q$ by $q^2$ and then setting $a=-q,\ b=q^2,\ c=-q^3$ and $z=-q^2$ in the following transformation \cite[p.~359, (III.2)]{gasper}
\begin{align*}
{}_2\phi_1(a,b;c;q,z)=\frac{(c/b;q)_\infty(bz;q)_\infty}{(c;q)_\infty(z;q)_\infty}{}_2\phi_1(abz/c,b;bz;q,c/b),
\end{align*}
we deduce that
\begin{align}
(1+q)\sum_{n=0}^\infty\frac{(-1)^nq^{2n}}{1+q^{2n+1}}&=\frac{(-q;q^2)_\infty(-q^4;q^2)_\infty}{(-q^3;q^2)_\infty(-q^2;q^2)_\infty}\sum_{n=0}^\infty\frac{(-q^2;q^2)_n}{(-q^4;q^2)_n}(-q)^n.\nonumber
\end{align}
This simplifies to
\begin{align}\label{2phi12}
\sum_{n=0}^\infty\frac{(-1)^nq^{2n}}{1+q^{2n+1}}&=\sum_{n=0}^\infty\frac{(-1)^nq^{n}}{1+q^{2n+2}}.
\end{align}
Now we first replace $q$ by $q^2$ and then let $a=-q^2,\ b=-q^4$ and $z=-q$ in Ramanujan's ${}_1\psi_1$ summation formula \cite[p.~239, (II 29)]{gasper}
\begin{align}\label{ram1psi1}
\sum_{n=-\infty}^\infty\frac{(a;q)_n}{(b;q)_n}z^n=\frac{(az,q)_\infty(q/(az);q)_\infty(q;q)_\infty(b/a;q)_\infty}{(z,q)_\infty(b/(az);q)_\infty(b;q)_\infty(q/a;q)_\infty},
\end{align}
so as to obtain
\begin{align}\label{1psi1}
\sum_{n=-\infty}^\infty\frac{(-q^2;q^2)_n}{(-q^4;q^2)_n}(-q)^n\nonumber&=\frac{(q^3;q^2)_\infty(q^{-1};q^2)_\infty(q^2;q^2)_\infty^2}{(-q;q^2)_\infty^2(-q^4;q^2)_\infty(-1;q^2)_\infty}\nonumber\\
&=-\frac{1+q^2}{2q}\frac{(q;q^2)_\infty^2(q^2;q^2)_\infty^2}{(-q;q^2)_\infty^2(-q^2;q^2)_\infty^2}\nonumber\\
&=-\frac{1+q^2}{2q}\frac{(q;q)_\infty^2}{(-q;q)_\infty^2}.
\end{align}
Note that the series on the left hand side can be rewritten as
\begin{align}\label{1psi2}
\sum_{n=-\infty}^\infty\frac{(-q^2;q^2)_n}{(-q^4;q^2)_n}(-q)^n&=\sum_{n=0}^\infty\frac{(-q^2;q^2)_n}{(-q^4;q^2)_n}(-q)^n+\sum_{n=1}^\infty\frac{(-q^2;q^2)_{-n}}{(-q^4;q^2)_{-n}}(-q)^n\nonumber\\
&=(1+q^2)\sum_{n=0}^\infty\frac{(-1)^nq^{n}}{1+q^{2n+2}}+\frac{(1+q^2)}{q^2}\sum_{n=1}^\infty\frac{(-1)^nq^{n}}{1+q^{2n-2}}\nonumber\\
&=(1+q^2)\sum_{n=0}^\infty\frac{(-1)^nq^{n}}{1+q^{2n+2}}-\frac{1+q^2}{2q}+\frac{(1+q^2)}{q^2}\sum_{n=2}^\infty\frac{(-1)^nq^{n}}{1+q^{2n-2}}\nonumber\\
&=2(1+q^2)\sum_{n=0}^\infty\frac{(-1)^nq^{n}}{1+q^{2n+2}}-\frac{1+q^2}{2q},
\end{align}
Hence, \eqref{1psi1} and \eqref{1psi2} imply
\begin{align}\label{final1729}
\sum_{n=0}^\infty\frac{(-1)^nq^{n}}{1+q^{2n+2}}=\frac{1}{4q}-\frac{1}{4q}\frac{(q;q)_\infty^2}{(-q;q)_\infty^2}.
\end{align}
Lemma now follows from \eqref{2phi12} and \eqref{final1729}.
\end{proof}

\begin{proof}[\textbf{Theorem \textup{\ref{spt theorem}}}][]
Letting $a=1$, $r_1=z$, $r_2=1/z$ and taking $b\to\infty$ in \eqref{andrewsstacks} and using the definition of ${}_{10}\phi_9$ and the fact that $$\lim_{b\to\infty}\frac{(b)_n}{b^n}=(-1)^nq^{n(n-1)/2},$$ we deduce that
\begin{align}\label{1step}
&\lim_{N\to\infty}\sum_{n=0}^\infty\frac{(-q)_n(z)_n(z^{-1})_n(-z)_n(-z^{-1})_n(q^{-N})_n(-q^{-N})_n}{(-1)_n(q/z)_n(-q/z)_n(qz)_n(-qz)_n(q^{N+1})_n(-q^{N+1})_n}q^{2Nn+3n+\frac{n(n-1)}{2}}\nonumber\\
&=\frac{(q^2;q^2)_\infty^2}{(q^2/z^2;q^2)_\infty(q^2z^2;q^2)_\infty}\sum_{n=0}^\infty\frac{(z^2;q^2)_n(z^{-2};q^2)_nq^{2n}}{(q^2;q^2)_n(-q;q)_{2n}}.
\end{align}
By doing simple manipulations, one can see that
\begin{align}
\lim_{N\to\infty}\frac{(q^{-N})_n(-q^{-N})_n}{(q^{N+1})_n(-q^{N+1})_n}q^{2nN}=(-1)^nq^{n(n-1)}.\nonumber
\end{align}
Substituting it in \eqref{1step} and simplifying, we are led to
\begin{align}
&\frac{(q^2;q^2)_\infty^2}{(q^2/z^2;q^2)_\infty(q^2z^2;q^2)_\infty}\sum_{n=0}^\infty\frac{(z^2;q^2)_n(z^{-2};q^2)_nq^{2n}}{(q^2;q^2)_n(-q;q)_{2n}}\nonumber\\
&=\sum_{n=0}^\infty\frac{(-1)^n(1+q^n)(1-z^2)(1-z^{-2})}{(1-z^2q^{2n})(1-q^{2n}z^{-2})}q^{3n+\frac{3}{2}n(n-1)}.\nonumber
\end{align}
Replacing $z^2$ by $z$ and rearranging terms, we obtain
\begin{align}
&\sum_{n=0}^\infty\frac{(z;q^2)_n(z^{-1};q^2)_nq^{2n}}{(q^2;q^2)_n(-q;q)_{2n}}\nonumber\\
&=\frac{1}{(q^2;q^2)_\infty^2}\left((z^{-1}q^2;q^2)_\infty(zq^2;q^2)_\infty+\sum_{n=1}^\infty\frac{(-1)^n(1+q^n)(z^{-1};q^2)_\infty(z;q^2)_\infty}{(1-zq^{2n})(1-q^{2n}z^{-1})}q^{\frac{3}{2}n(n+1)}\right).\nonumber
\end{align}
We now differentiate both sides twice with respect to $z$, where we use \cite[Equation (2.1)]{andrewsspt}
\begin{align*}
-\frac{1}{2}\frac{d^2}{dz^2}\left((1-z)(1-z^{-1})f(z)\right)=f(1),
\end{align*}
first with $f(z)=(zq^2,q^2)_{n-1}(z^{-1}q^2,q^2)_{n-1}^2$, resulting in
\begin{align*}
-\frac{1}{2}\frac{d^2}{dz^2}\left((z^{-1};q^2)_n(z;q^2)_n\right)\Big|_{z=1}=(q^2;q^2)_{n-1}^2,
\end{align*}
and then with $$f(z)=\frac{(zq^2,q^2)_{\infty}(z^{-1}q^2,q^2)_{\infty}^2}{(1-zq^{2n})(1-q^{2n}z^{-1})},$$ resulting in
\begin{align*}
-\frac{1}{2}\frac{d^2}{dz^2}\left(\frac{(z^{-1};q^2)_\infty(z;q^2)_\infty}{(1-zq^{2n})(1-q^{2n}z^{-1})}\right)\Big|_{z=1}=\frac{(q^2;q^2)_{\infty}^2}{(1-q^{2n})^2},
\end{align*}
and \cite[Equation (2.4)]{andrewsspt}
\begin{align*}
-\frac{1}{2}\frac{d^2}{dz^2}\left((z^{-1}q^2;q^2)_\infty(zq^2;q^2)_\infty\right)\Big|_{z=1}=(q^2;q^2)_\infty^2\sum_{n=1}^\infty\frac{nq^{2n}}{1-q^{2n}},
\end{align*}
so as to obtain
\begin{align}\label{almost spt id}
\sum_{n=0}^\infty\frac{(q^2;q^2)_{n-1}^2q^{2n}}{(q^2;q^2)_n(-q;q)_{2n}}=\sum_{n=1}^\infty\frac{nq^{2n}}{1-q^{2n}}+\sum_{n=1}^\infty\frac{(-1)^n(1+q^n)q^{\frac{3}{2}n(n+1)}}{(1-q^{2n})^2}.
\end{align}
Note that the left-hand side of the above expression can be rewritten as
\begin{align*}
\sum_{n=0}^\infty\frac{(q;q)_{n}q^{2n}}{(1-q^{2n})^2(-q^{n+1};q)_{n}}.
\end{align*}
Now using the above fact in \eqref{almost spt id} then dividing both sides by $(q;q)_\infty$, we complete the proof of the theorem.
\end{proof}

\section{A new representation of the Mock Theta function $f_3(q)$}\label{new representation for f}

\begin{proof}[\textbf{Theorem \textup{\ref{new representation}}}][]
Note that
\begin{align}\label{transf}
\sum_{n=1}^\infty\frac{q^{2n}}{(q^{2n};q^2)_{n+1}(q^{2n+1};q)_\infty}
&=\sum_{n=1}^\infty\frac{q^{2n}}{(1-q^{2n})(-q^{n+1};q)_n(q^{n+1};q)_\infty}\nonumber\\
&=\sum_{n=1}^\infty\frac{q^{2n}}{(-q^{n};q)_{n+1}(q^{n};q)_\infty}\nonumber\\
&=q^2\sum_{n=0}^\infty\frac{q^{2n}}{(-q^{n+1};q)_{n+2}(q^{n+1};q)_\infty}\nonumber\\
&=\frac{q^2}{(q;q)_\infty}\sum_{n=0}^\infty\frac{(q;q)_n q^{2n}}{(-q^{n+1};q)_{n+2}}\nonumber\\
&=\frac{q^2}{(q;q)_\infty}\sum_{n=0}^\infty\frac{(q;q)_n(-q;q)_n q^{2n}}{(-q;q)_{2n+2}}\nonumber\\
&=\frac{q^2}{(q;q)_\infty(1+q)(1+q^2)}\sum_{n=0}^\infty\frac{(q^2;q^2)_n q^{2n}}{(-q^3;q^2)_{n}(-q^4;q^2)_{n}}.
\end{align}
We first replace $q$ by $q^2$ in \eqref{4parameter} and then  let $a=q,\ b=q^2$, $B=q^2$ and $A\to0$ to obtain
\begin{align}\label{4para1}
\sum_{n=0}^\infty\frac{(q^2;q^2)_n q^{2n}}{(-q^3;q^2)_{n}(-q^4;q^2)_{n}}&=-\frac{1}{q}\frac{(q^2;q^2)_\infty}{(-q^4;q^2)_\infty(-q^3;q^2)_\infty}\sum_{m=0}^\infty\frac{(-1)^m q^{m(m-1)}q^{3m}}{(-q;q^2)_{m+1}}\nonumber\\
&\qquad+(1+q^2)\sum_{m=0}^\infty\frac{(-q^{-1};q^2)_{m+1}(-q^2)^m}{(-q;q^2)_{m+1}}\nonumber\\
&=-\frac{1}{q^2}\frac{(q^2;q^2)_\infty}{(-q^4;q^2)_\infty(-q^3;q^2)_\infty}\sum_{m=1}^\infty\frac{(-1)^{m-1} q^{m^2}}{(-q;q^2)_{m}}\nonumber\\
&\qquad+\frac{(1+q)(1+q^2)}{q}\sum_{m=0}^\infty\frac{(-1)^mq^{2m}}{1+q^{2m+1}}.
\end{align}
Equations \eqref{transf} and \eqref{4para1} together yield
\begin{align}\label{2sums}
&\sum_{n=1}^\infty\frac{q^{2n}}{(q^{2n};q^2)_{n+1}(q^{2n+1};q)_\infty}\nonumber\\
&=\frac{q^2}{(q;q)_\infty(1+q)(1+q^2)}\left\{-\frac{1}{q^2}\frac{(q^2;q^2)_\infty}{(-q^4;q^2)_\infty(-q^3;q^2)_\infty}\sum_{m=1}^\infty\frac{(-1)^{m-1} q^{m^2}}{(-q;q^2)_{m}}\right.\nonumber\\
&\left.\qquad+\frac{(1+q)(1+q^2)}{q}\sum_{m=0}^\infty\frac{(-1)^mq^{2m}}{1+q^{2m+1}}\right\}\nonumber\\
&=-\sum_{m=1}^\infty\frac{(-1)^{m-1} q^{m^2}}{(-q;q^2)_{m}}+\frac{1}{(q;q)_\infty}\sum_{m=0}^\infty\frac{(-1)^mq^{2m+1}}{1+q^{2m+1}}
\end{align}
From \cite[Chapter 2, Entry 2.3.9]{rlnb-V}, we have
\begin{align}\label{239}
\sum_{m=0}^\infty\frac{(-1)^{m} q^{m^2}}{(-aq^2;q^2)_{m}}=(1+a)\sum_{m=1}^\infty\frac{(-1)^{m-1} q^{m^2}}{(-aq;q^2)_{m}}+\frac{\varphi(-q)}{(-aq;q)_\infty},
\end{align}
where $$\varphi(-q):=\sum_{n=-\infty}^\infty(-1)^nq^{n^2}=\frac{(q;q)_\infty}{(-q;q)_\infty}.$$
Letting $a=1$ in the above equation and simplifying, we see that
\begin{align}\label{phi312}
\sum_{m=1}^\infty\frac{(-1)^{m-1} q^{m^2}}{(-q;q^2)_{m}}=\frac{1}{2}\phi_3(-q)-\frac{1}{2}\frac{(q;q)_\infty}{(-q;q)_\infty^2},
\end{align}
where $\phi_3(q)$ is another mock theta function of the third order
\begin{align}
\phi_3(q):=\sum_{n=0}^\infty\frac{q^{n^2}}{(-q^2;q^2)_n}.\nonumber
\end{align}
We need the following relation between mock theta functions $f_3(q)$ and $\phi_3(q)$ \cite[Chapter 2, Entry 2.3.1]{rlnb-V}
\begin{align}\label{2.3.1}
\phi_3(-q)=\frac{1}{2}f_3(q)+\frac{1}{2}\frac{(q;q)_\infty}{(-q;q)_\infty^2}
\end{align}
Substituting value from \eqref{2.3.1} in \eqref{phi312}, we deduce that
\begin{align}\label{suminf}
\sum_{m=1}^\infty\frac{(-1)^{m-1} q^{m^2}}{(-q;q^2)_{m}}=\frac{1}{4}f_3(q)-\frac{1}{4}\frac{(q;q)_\infty}{(-q;q)_\infty^2}.
\end{align}
Equations \eqref{2sums} and \eqref{suminf} imply that
\begin{align*}
&\sum_{n=1}^\infty\frac{q^{2n}}{(q^{2n};q^2)_{n+1}(q^{2n+1};q)_\infty}=-\frac{1}{4}f_3(q)+\frac{1}{4}\frac{(q;q)_\infty}{(-q;q)_\infty^2}+\frac{1}{(q;q)_\infty}\sum_{m=0}^\infty\frac{(-1)^mq^{2m+1}}{1+q^{2m+1}}.
\end{align*}
Now the claimed result follows upon invoking Lemma \ref{closed form lemma} in the above equation.
\end{proof}

\section{Partition ranks and two-color partitions}\label{partition ranks and two-color}
\begin{proof}[\textbf{Theorem \textup{\ref{G(n)=N(n)}}}][]
Note that 
\begin{align}\label{g1}
\sum_{n=1}^\infty\frac{q^{2n}}{(1-q^{2n})(-q^{n+1};q)_n(q^{n+1};q)_\infty}&=\sum_{n=1}^\infty\frac{q^{2n}}{(1-q^{2n})(-q^{n+1};q)_n(q^{n+1};q)_n(q^{2n+1};q)_\infty}\nonumber\\
&=\sum_{n=1}^\infty\frac{q^{2n}}{(1-q^{2n})(q^{2n+2};q^2)_n(q^{2n+1};q)_\infty}\nonumber\\
&=\sum_{n=1}^\infty\frac{q^{2n}}{(q^{2n+2};q^2)_n(q^{2n};q)_\infty}\nonumber\\
&=\sum_{n=1}^\infty G(n)q^n,
\end{align}
where the last step follows easily after observing that the expression in the penultimate step is the generating function of $G(n)$.

We have
\begin{align*}
\sum_{n=1}^\infty p(n)q^n=\frac{1}{(q;q)_\infty},
\end{align*}
where $p(n)$ denotes the number of partitions of $n$. We know that 
\begin{align*}
f_3(q)=1+\sum_{n=1}^\infty \left(N_e(n)-N_o(n)\right)q^n,
\end{align*}
where $N_e(n)$ and $N_o(n)$ denotes the number of partitions of $n$ with even rank and odd rank, respectively. Moreover, we have the identities 
$$p(n)=N_e(n)+N_o(n)\qquad \textup{and}\qquad  N_o(n)=2N_o^+(n).$$ Therefore, employing the above facts along with Theorem \ref{new representation}, we conclude that
\begin{align}\label{g2}
\sum_{n=1}^\infty N_o^+(n)q^n&=\sum_{n=1}^\infty\frac{q^{2n}}{(q^{2n};q^2)_{n+1}(q^{2n+1};q)_\infty}\nonumber\\
&=\sum_{n=1}^\infty\frac{q^{2n}}{(1-q^{2n})(-q^{n+1};q)_n(q^{n+1};q)_\infty}.
\end{align}
Equations \eqref{g1} and \eqref{g2} together now prove our result.
\end{proof}

\begin{proof}[\textbf{Corollary \textup{\ref{Ne(n)}}}][]
Using \eqref{g2} and Theorem \ref{new representation}, it is easy to prove \eqref{even rank generating function}.

Equation \eqref{fine numbers result} follows from the following result \cite[p.~56, Equation (26.27)]{fine}
\begin{align}
f_3(q)=\frac{1}{(q;q)_\infty}\left\{1+4\sum_{n=1}^\infty J(n)q^n\right\},\nonumber
\end{align}
and after invoking Theorem \ref{new representation}.

\end{proof}

\section{Two-color partitions with odd smallest part}\label{odd smallest part}
Let $G'(n)$ denotes the number of two-color partitions with color red and blue, where the smallest part is odd (say $(2m+1)$)  and all the red parts are even and lie in the interval $(2m,4m]$. Then, it is clear that
\begin{align}\label{g'}
\sum_{n=0}^\infty G'(n)q^n=\sum_{n=0}^\infty\frac{q^{2n+1}}{(q^{2n+1};q)_\infty(q^{2n+2};q^2)_n}.
\end{align}

Let $\phi_3(q)$ is another mock theta function of the third order
\begin{align}\label{phi3}
\phi_3(q):=\sum_{n=0}^\infty\frac{q^{n^2}}{(-q^2;q^2)_n}.
\end{align}

\begin{theorem}\label{with odd smallest part}
Let $G'(n)$ and  $\phi_3(n)$ be defined as in \eqref{g'} and \eqref{phi3}, respectively. Then, we have
\begin{align*}
\sum_{n=0}^\infty G'(n)q^n&=q^2-q(1+q)\phi_3(-q)+\frac{q(3-q)}{2(q;q)_\infty}+\frac{q(1+q)(q^2;q^2)_\infty}{2(-q;q)_\infty^3}.
\end{align*}
\end{theorem}
\begin{proof}
Note that
\begin{align*}
\sum_{n=0}^\infty G'(n)q^n&=\frac{q}{(q;q)_\infty}\sum_{n=0}^\infty\frac{(q;q)_{2n}q^{2n}}{(q^{n+1};q)_n(-q^{n+1};q)_n}\\
&=\frac{q}{(q;q)_\infty}\sum_{n=0}^\infty\frac{(q;q)_{2n}(q;q)_n(-q;q)_nq^{2n}}{(q;q)_n(q^{n+1};q)_n(-q;q)_n(-q^{n+1};q)_n}\\
&=\frac{q}{(q;q)_\infty}\sum_{n=0}^\infty\frac{(q^2;q^2)_nq^{2n}}{(-q;q)_{2n}}\\
&=\frac{q}{(q;q)_\infty}\sum_{n=0}^\infty\frac{(q^2;q^2)_nq^{2n}}{(-q;q^2)_{n}(-q^2;q^2)_{n}}.
\end{align*}
Employing Lemma \ref{analytic continuation} with letting $b=1$ in the above equation, we arrive at
\begin{align}\label{seriesf}
\sum_{n=0}^\infty G'(n)q^n&=\frac{q}{(q;q)_\infty}\left\{-\frac{q(q^2;q^2)_\infty}{(-q;q^2)_\infty(-q^2;q^2)_\infty}\sum_{m=0}^\infty\frac{(-1)^mq^{m^2+2m}}{(-q^3;q^2)_{m+1}}+\frac{1+q}{1+q^3}\right.\nonumber\\
&\qquad\left.+(1+q)(1-q^2)\sum_{m=1}^\infty\frac{(-1)^mq^{2m+1}}{(1+q^{2m+1})(1+q^{2m+3})}\right\}\nonumber\\
&=q\sum_{m=1}^\infty\frac{(-1)^mq^{m^2}}{(-q^3;q^2)_{m}}+\frac{q(1+q)}{(q;q)_\infty(1+q^3)}\nonumber\\
&\qquad+\frac{q(1+q)(1-q^2)}{(q;q)_\infty}\sum_{m=1}^\infty\frac{(-1)^mq^{2m+1}}{(1+q^{2m+1})(1+q^{2m+3})},
\end{align}
where in the last step we replaced $m\to m-1$ in the first series.

We set $a=q^2$ in \eqref{239} and rearrange terms so that
\begin{align*}
\sum_{m=1}^\infty\frac{(-1)^mq^{m^2}}{(-q^3;q^2)_{m}}=\frac{(q;q)_\infty}{(-q;q)_\infty(-q^2;q)_\infty}-\sum_{m=0}^\infty\frac{(-1)^mq^{m^2}}{(-q^2;q^2)_{m+1}}.
\end{align*}

Substituting this value in \eqref{seriesf}, we obtain
\begin{align}\label{beforephi}
\sum_{n=0}^\infty G'(n)q^n&=\frac{q(q;q)_\infty}{(-q;q)_\infty(-q^2;q)_\infty}-q\sum_{m=0}^\infty\frac{(-1)^mq^{m^2}}{(-q^2;q^2)_{m+1}}+\frac{q(1+q)}{(q;q)_\infty(1+q^3)}\nonumber\\
&\qquad+\frac{q(1+q)(1-q^2)}{(q;q)_\infty}\sum_{m=1}^\infty\frac{(-1)^mq^{2m+1}}{(1+q^{2m+1})(1+q^{2m+3})}.
\end{align}
Observe that
\begin{align}\label{phi3m}
\sum_{m=0}^\infty\frac{(-1)^mq^{m^2}}{(-q^2;q^2)_{m+1}}&=\sum_{m=0}^\infty\frac{(-1)^mq^{m^2}}{(-q^2;q^2)_{m}}\left(\frac{1}{1+q^{2m+2}}-1+1\right)\nonumber\\
&=-\sum_{m=0}^\infty\frac{(-1)^mq^{m^2+2m+2}}{(-q^2;q^2)_{m+1}}+\phi_3(-q)\nonumber\\
&=q\sum_{m=1}^\infty\frac{(-1)^mq^{m^2}}{(-q^2;q^2)_{m}}+\phi_3(-q)\nonumber\\
&=q\left\{-1+\phi_3(-q)\right\}+\phi_3(-q)\nonumber\\
&=-q+(1+q)\phi_3(-q).
\end{align}
Equations \eqref{beforephi} and \eqref{phi3m} together imply
\begin{align}\label{last2}
\sum_{n=0}^\infty G'(n)q^n&=q^2-q(1+q)\phi_3(-q)+\frac{q(1+q)(q;q)_\infty}{(-q;q)_\infty^2}+\frac{q(1+q)}{(q;q)_\infty(1+q^3)}\nonumber\\
&\qquad+\frac{q(1+q)(1-q^2)}{(q;q)_\infty}\sum_{m=1}^\infty\frac{(-1)^mq^{2m+1}}{(1+q^{2m+1})(1+q^{2m+3})}.
\end{align}
We now evaluate the series involved in the above expression. Replacing $q$ by $q^2$ and letting $a=-q,\ b=-q^5$ and $z=-q^2$ in the Ramanujan's ${}_1\psi_1$ formula \eqref{ram1psi1}, we obtain
\begin{align}\label{last}
\sum_{n=-\infty}^\infty\frac{(-1)^nq^{2n}(-q;q^2)_n}{(-q^5,q^2)_n}&=\frac{(q^3;q^2)_\infty(q^{-1};q^2)_\infty(q^2;q^2)_\infty(q^4;q^2)_\infty}{(-q^2,q^2)_\infty^2(-q^5;q^2)_\infty(-q;q^2)_\infty}\nonumber\\
&=-\frac{(1+q)(1+q^3)}{q(1-q^2)}\frac{(q;q^2)_\infty^2(q^2;q^2)_\infty^2}{(-q^2,q^2)_\infty^2(-q;q^2)_\infty^2}\nonumber\\
&=-\frac{(1+q)(1+q^3)}{q(1-q^2)}\frac{(q;q)_\infty^2}{(-q,q)_\infty^2}.
\end{align}
It is easy to see that the series on the left-hand side can be written as
\begin{align*}
\sum_{n=-\infty}^\infty\frac{(-1)^nq^{2n}(-q;q^2)_n}{(-q^5,q^2)_n}=2(1+q)(1+q^3)\sum_{n=0}^\infty\frac{(-1)^nq^{2n}}{(1+q^{2n+1})(1+q^{2n+3})}-\frac{1}{q}\frac{1+q^3}{1+q}.
\end{align*} 
This implies that
\begin{align}\label{last1}
&\sum_{n=1}^\infty\frac{(-1)^nq^{2n}}{(1+q^{2n+1})(1+q^{2n+3})}\nonumber\\
&=\frac{1}{2q(1+q)^2}-\frac{1}{(1+q)(1+q^3)}+\frac{1}{2(1+q)(1+q^3)}\sum_{n=-\infty}^\infty\frac{(-1)^nq^{2n}(-q;q^2)_n}{(-q^5,q^2)_n}\nonumber\\
&=\frac{1}{2q(1+q)^2}-\frac{q}{(1+q)(1+q^3)}-\frac{1}{2q(1-q^2)}\frac{(q;q)_\infty^2}{(-q,q)_\infty^2},
\end{align}
where the last step follows from \eqref{last}.  

Now result follows from equations \eqref{last2} and \eqref{last1}.
\end{proof}

\section{Concluding Remarks}\label{concluding remarks}
One of our main results shows that the number of partitions of a natural number with positive odd rank is equal to the number of two-color partitions (red and blue), where the smallest part is even (say $2n$) and all red parts are even and lie within the interval $(2n,4n]$. We proved this result through analytic methods. Now this naturally raises a combinatorial question:

 \emph{Can one produce a bijection between these two sets of partitions}?

Another direction worth exploring is the following: It is evident that the left-hand side of the $q$-series identity in \eqref{spt theorem eqn} serves as the generating function for the smallest parts function associated with the partitions generated by $G(n)$. This new analogue of the smallest part partition function certainly merits further investigation.

\medskip

\noindent {\bf{Acknowledgements:}}\, Authors would like to thank Aa Ja Yee and Atul Dixit for helpful discussions. The first author is partially supported by the Simons Foundation Grant 633284, and the second author is partially supported by the FIG grant of IIT Roorkee. Both authors sincerely thank these institutions for their support.

\end{document}